\documentclass[12pt,reqno,oneside]{amsart}
\usepackage{amsmath}
\usepackage{amssymb}
\usepackage{amsthm}
\usepackage{graphics}
\usepackage{eucal}
\usepackage{bm}
\usepackage[colorlinks]{hyperref}
\usepackage[capitalize,nameinlink,noabbrev]{cleveref}
\usepackage{mathrsfs}
\usepackage[normalem]{ulem}
\usepackage{color}
\usepackage[dvipsnames]{xcolor}
\usepackage{mathtools}
\usepackage{geometry}
\newtheorem{theorem}{Theorem}[section]

\newtheorem{proposition}{Proposition}[section]
\newtheorem{lemma}{Lemma}[section]

\newtheorem{problem}{Problem}[section]

\numberwithin{equation}{section}
\newcommand\spr{Sprind\v{z}uk}

\newcommand{\J}{\mathbf{J}}
\newcommand{\h}{\mathbf{h}}

\newcommand{\diag}{\operatorname{diag}}

\newcommand{\bq}{\mathbf{q}}
\newcommand{\bx}{\mathbf{x}}
\newcommand{\bU}{\mathbf{U}}

\newcommand{\bp}{\mathbf{p}}

\newcommand{\bv}{\mathbf{v}}
\newcommand{\bw}{\mathbf{w}}
\newcommand{\f}{\mathbf{f}}
\newcommand{\ff}{\mathit{f}}

\newcommand{\ba}{\mathbf{a}}

\newcommand{\by}{\mathbf{y}}

\newcommand{\Z}{\mathbb{Z}}
\newcommand{\R}{\mathbb{R}}
\newcommand{\N}{\mathbb{N}}

\newcommand{\bV}{\mathbf{V}}

\newcommand{\cM}{\mathcal{M}}

\usepackage{mathtools}
\usepackage{tikz-cd}
\usepackage{graphicx}

\begin{document}

\title{Quantitative Khintchine in Simultaneous approximation}

\author{Shreyasi Datta}
\address{\textbf{Shreyasi Datta} \\
Department of Mathematics, University of Michigan, Ann Arbor, MI 48109-1043}
\email{dattash@umich.edu}

\date{}
\begin{abstract}
  In a ground-breaking work \cite{BY}, Beresnevich and Yang recently proved Khintchine's theorem in simultaneous Diophantine approximation for nondegenerate manifolds resolving a long standing problem in the theory of Diophantine approximation. In this paper, we prove an effective version of their result. 
\end{abstract}
\maketitle
\section{Introduction}
We are motivated by recent development in simultaneous Diophantine approximation problems in \cite{BY}. Let $\Psi:(0,+\infty)^m\to(0,1)$ be a function. We call a matrix $A\in M_{n\times m}$ to be $\Psi$-approximable if there are infinitely many $(\bp,\bq)\in \Z^n\times \Z^m\setminus 0$ such that \begin{equation}\label{Khintchine}
    \Vert A\bq\Vert_{\Z}\leq \Psi(\bq).
\end{equation} Here and elsewhere we denote the set of all $n$ rows and $m$ columns matrices by $M_{n\times m}$. We denote the sup norm and the $l^2$ norm in $\R^n$ by $\Vert \cdot\Vert_\infty$ and $\Vert \cdot\Vert$, respectively and supremum distance of a point from its nearest integer vector by $\Vert \cdot\Vert_{\Z}$. Khintchine's convergence theorem (\cite{KHIN,Groshev} asserts that the set of $\Psi$-approximable matrices has Lebesgue measure $0$ if $\sum \Psi(\bq)^n<\infty.$ There is a converse of this theorem with more restrictions on the approximating function $\Psi$, which is referred as Khintchine's divergence  theorem. When $n=1$, the type of approximation in \cref{Khintchine} is named as dual approximation and if $m=1$ it is named as simultaneous Diophantine approximation.

Khintchine's theorem for nondegenerate manifolds (see \S \cref{setup}) in dual setting was proved in \cite{BKM, BBKM} which originated from a ground-breaking work \cite{KM} of Kleinbock and Margulis on \textit{Baker--{\spr} conjecture}. An attempt to prove simultaneous Khintchine's theorem on manifolds faces significant difficulties as it deals with notoriously hard question of counting rational points near a manifold. The divergence part of simultaneous Khintchine for $C^3$ planar curves was resolved in \cite{BDV07}, and for any analytic manifolds it was settled by Beresnevich in \cite{Ber12}. For recent developments in this direction for curves (without analyticity assumption), readers are referred to \cite{BLVV17, BVVZ21}.

\indent In the convergence side of simultaneous Khintchine theorem for manifolds, a breakthrough  came in \cite{VV06} by Vaughan and Velani, where they resolved this case for all $C^3$ nondegenerate curves. There have been many works \cite{BLVV17,Hua15, DRV91, Ber77} that dealt with convergence part for special manifolds with restrictions on dimension and geometry of those manifolds. For affine subspaces with suitable \textit{exponent} condition, both convergence and divergence is settled in the latest work of Huang \cite{Hua22}. \\
\indent Despite many achievements, even for curves of dimension $3$ and more, the covergence case remained open until recently. In a very recent work \cite{BY}, this long-standing problem of Khintchine's theorem in simultaneous Diophantine approximation for nondegenerate manifolds is settled by Beresnevich and Yang, almost twenty years after the settlement of its dual counterpart in \cite{BKM}. The following is the main theorem in \cite{BY}.
\begin{theorem}[Theorem $1.2$ in \cite{BY}]\label{simkhinttheorem}
Let $\psi:(0,+\infty)\to(0,1)$ be monotonic. Let $n\geq 2$, $\cM\subset\R^n$ be a nondegenerate submanifold, and $$\sum_{q=1}^\infty \psi^n(q)<\infty.$$ Then almost all points on $\cM$ are not $\psi$-approximable.
\end{theorem}
In this paper we prove an effective version of the above theorem. Our work is along the same theme of works as in \cite{ABLVZ} (and \cite{Ganguly-Ghosh16}) where an effective version of dual Khintchine's theorem for nondegenerate manifolds (and affine subspaces) was proved. Let us be specific about what we mean by effective.
As a corollary of \cref{simkhinttheorem}, one gets that for almost every
$\by\in\cM$, there exists $\kappa(\by)>0$ such that \begin{equation}\label{effec}
\Vert q\by\Vert_{\Z}>\kappa(\by)\psi(q),\end{equation} for all $q\in\Z\setminus\{0\}$, with the assumption on $\psi$ as in \cref{simkhinttheorem}. One can ask if it is possible to get rid of dependency of $\by$ in $\kappa(\by)$, and if yes, then for which vectors.
The following problem was stated in \cite{ABLVZ} in any set-up where convergence Khintchine's theorem is known.
\begin{problem}\label{problem}
Investigate the dependency between $\kappa>0$ and the probability of the set of vectors $\by$ such that $\kappa(\by)=\kappa.$
\end{problem}
 
 \indent Another way of writing \cref{effec} is that with the assumptions in \cref{simkhinttheorem}, for any $1>\delta>0$ there is a constant $\kappa>0$ depending only on $\cM,\psi,\delta$ such that $$
\lambda_{\cM}(\{\by\in\cM~|~\Vert q\by\Vert_{\Z}>\kappa \psi(q) \text{ for all } q\in \Z\setminus\{0\}\})\geq 1-\delta,$$
where $\lambda_{\cM}$ is a Haar measure on a compact subset of the manifold. An interpretation of \cref{problem} is to find dependency of $\kappa$ on $\delta.$ In this paper, we undertake this problem for nondegenerate manifolds in the context of simultaneous Diophantine approximation.\\
In the past few years, the study of the achievable degrees of freedom in various schemes on Interference Alignment from electronics communication requires one to understand \cref{problem}; see \cite{MGAK, BVdummies}. In particular, readers are referred to the works \cite{ABLVZ,Jafar, GMK, MGAK, MMK, NW, WSV}, and \cite{BVdummies} for more survey. We expect that our result will have application in this direction.
\subsection{Set-up}\label{setup}
Without loss of generality, we assume that the manifolds of our interest are images of $\f:\bV:=3^{n+d+3}\bU\to\R^n$, where $\bU$ is an open subset of $\R^d$. Moreover, using implicit function theorem, we assume without loss of generality that $$\cM=\{\f(\bx)~|~\bx\in \bV\},$$ where $\f(\bx)=(\bx,\ff_1(\bx),\cdots,\ff_m(\bx)), n=m+d$.
Nondegeneracy for an analytic manifold means that the manifold is not contained inside any affine subspace, even locally. To be precise, we recall the definition from \cite{KM}. A map $\f:\bV\to \R^n$ is called $l$-nondegenerate at $\bx_1\in \bV$ if $\f$ is $l$-times continuously differentiable on a neighborhood of $\bx_1$ and the partial derivatives of $\f$ at $\bx_1$ of orders up to $l$ spans $\R^n$. We say $\f$ is nondegenerate at $\bx_1$ if it is $l$-nondegenerate at $\bx_1$ for some $l\in \N$. Here and elsewhere, $\lambda_d$ denotes the Lebesgue measure on $\R^d$.\\ 
\indent Throughout the paper, we assume that $\bx_0$ is the center of $\bU$, and radius of this ball is $r(\bU).$ Without loss of generality, we assume that $\f$ is actually defined on the closure of $3^{n+d+2}\bU$. We also assume that $\f$ is $l$-nondegenerate everywhere on the closure of $3^{n+d+2}\bU$. We assume $r(\bU)$ satisfies \cite[Equation 46]{ABLVZ}. These assumptions are nonrestrictive for our main theorem and  were also considered in \cite[Section 5]{ABLVZ}.
Since we are interested in nondegenerate manifolds, without loss of generality we assume the following:
\begin{equation}\label{Condition on derivative}
    \max_{1\leq k\leq d}\max_{1\leq i,j\leq d}\sup_{\bx\in\bU}\max\{\vert \partial_i\ff_k(\bx)\vert,\vert \partial^2_{i,j}\ff_k(\bx)\vert\}:= M>0.
\end{equation}

Since it is enough to prove the main theorem for $\max\{\psi(q), q^{-\frac{5}{4n}}\}$, without loss of generality we assume that for all $q\in\N$,
\begin{equation}\label{condition on si}
\psi(q)^n>q^{-\frac{5}{4}}.
\end{equation}

\subsection{Main theorem}
Let us define $$\mathcal{B}(\kappa,\psi):=\{\bx\in \bU~|~ \Vert q\f(\bx)\Vert_{\Z}>\kappa\psi(q) \text{ for all } q\in \Z\setminus \{0\}\}.$$
Our main theorem is an effective version of \cref{simkhinttheorem}.
\begin{theorem}\label{main}
Suppose that $n\geq 2$, $\f$ is a $C^{l+1}$ function, and $\bU$ is as in \S\cref{setup}. Let $\cM\subset\R^n$ be a compact manifold that is $l$-nondegenerate at every point, $\psi$ be monotonic and $\sum_{q=1}^\infty \psi(q)^n<\infty$. Let $$\begin{aligned}
\kappa_0:= {\min{\left(\frac{1}{2}, d_3^{-\frac{1}{4}}, e^{-\frac{n}{5}}d_3^{\frac{n}{5}}, \left(e^{-(n+1)}L_{\psi}^{-1}\right)^{\frac{1}{4n+5}}, d_3^{-(n+1)}e^{n}e^{\frac{5}{4}}, \left(\frac{1}{L_\psi e^{1+n}}\right)^{\frac{1}{5}} \right)}},
\end{aligned}$$ where $d_3$ is as in \cref{d3r}, and $L_\psi$ is as in \cref{trivialbound}.
Given any $0<\delta<1$, let $$\begin{aligned}   \kappa:= {\min{\left(\kappa_0, \left( \frac{\delta e^{nr+\frac{5r}{4}}}{2K_0c_r}\right)^{\frac{1}{r}}, \frac{\delta}{2^{n+3} c_4^{(n+1)}c_3^m c_2^d e^m e^{d+1}S_\psi} \right)}}^{\frac{5+4n}{n}},\end{aligned}$$ where $S_{\psi}, c'_1, c_2, c_3, c_4, r, c_r, K_0$ and $
E$ are as in \cref{S_psi,c_2,c_1,c_3,c_4,d3r,c_r,K_0 and E} and $M$ is as in \cref{Condition on derivative}. Then 
$$
\lambda_d(\mathcal{B}(\kappa,\psi))\geq (1-\delta)\lambda_d(\bU).$$
\end{theorem}

Our strategy of the proof is to divide our set of interest $\mathcal{B}(\kappa,\psi)$ into two parts, namely minor arc and major arc similar to \cite{BY}. But we need to do it in a more delicate manner. For instance, just replacing $\psi(q)$ by $\kappa\psi(q)$ to define arcs does not help, as it either makes the count of rational points in major arc very large or it makes the measure of minor arc large. To circumvent, we write $\kappa$ as a combination of two numbers $\eta_1<0$ and $\eta_2>0$ satisfying \eqref{kappacondition2}
and \eqref{kappacondition3}. Then we define minor arc and major arc in terms of $\eta_1$ and $\eta_2$; see \cref{definition of minor}. This phenomenon did not occur in effective dual Khintchine theorems for manifolds in \cite{ABLVZ, Ganguly-Ghosh16}. For the minor arc we use explicit measure estimates from \cite{ABLVZ}, and for the major arc we make the constants in \cite[Proposition 5.3]{BY} explicit.

We now give the precise values of the constants in the previous theorem. We define \begin{equation}\label{S_psi}
S_{\psi}:=\sum_{t=1}^\infty \psi(e^{t-1})^n e^{t-1},
\end{equation} 
\begin{equation}\label{c_1}
c_1':=(d+1)M,\end{equation}
\begin{equation}\label{c_2}
c_2:=(c'_1+1+d^2M), \end{equation}
\begin{equation}\label{c_3}
c_3:= c'_1+d^2M+d^2Mc_2, 
\end{equation}
\begin{equation}\label{c_4}
c_4:=\max\{c_3^{-1},c_2^{-1}, \frac{3}{2}\},
\end{equation}
\begin{equation}\label{d3r}
d_3:=(n+1)!^2(1+n+n^3M),r:=\frac{1}{d(2l-1)(n+1)}, \end{equation}
\begin{equation}\label{c_r}
\sum_{t\geq 1} e^{-r\frac{t}{4}}=c_r,
\end{equation}
\begin{equation}\label{K_0 and E}K_0:=d_3^{\frac{1}{d(2l-1)}}E(n+d+1)^{\frac{1}{2d(2l-1)}},\text{ and } 
E:= C(n+1)(3^dN_d)^{n+1}\rho^{\frac{-1}{d(2l-1)}},
\end{equation} where $N_d$ denotes the Besicovitch covering constant, $C$ can be found in \cite[Equation (52)]{ABLVZ}, $\rho$ is explicitly given in \cite[Equation (71)]{ABLVZ} and they only depend on $\f$ and $\bU$.

\section{Strategy of proof}

Note that $$ {\mathcal{B}(\kappa,\psi)}^c= \bigcup_{q\in \Z\setminus 0} \{\bx\in\bU~|~ \Vert q\f(\bx)\Vert_{\Z}<\kappa\psi(q)\}= \bigcup_{t\geq 1}\bigcup_{e^{t-1}
\leq q\leq e^{t}} \{\bx\in\bU~|~ \Vert q\f(\bx)\Vert_{\Z}<\kappa\psi(q)\}.$$ 
Let us define, 
$$\mathcal{M}(t,\kappa):=\left\{\bx\in\bU~|~ \left\Vert \f(\bx)-\frac{\bp}{q}\right\Vert_\infty<\kappa^{\frac{5+4n}{n}}\frac{\psi(e^{t-1})}{e^{t-1}}, e^{t-1}\leq q\leq e^t\right\}.$$
Then we have that $$ {\mathcal{B}(\kappa^{\frac{5+4n}{n}},\psi)}^c\subset \bigcup_{t\geq 1} \mathcal{M}(t,\kappa).$$

Suppose $\kappa<1$. Let \begin{equation}\label{kappacondition1}\eta_1=\frac{5\log\kappa}{n}<0, \text{ and } \eta_2=-4\log\kappa>0.\end{equation} 
Hence \begin{equation}\label{kappacondition2}\kappa=e^{n\eta_1+\eta_2},\end{equation} and \begin{equation}\label{kappacondition3}\kappa^{-1}=e^{n\eta_1+\frac{3}{2}\eta_2}.\end{equation} Following \cite{BY}, for any $t>0, 0<\varepsilon<1$, and $\Delta\subset\R^d$, we define 
\begin{equation}
    \mathcal{R}(\Delta;\varepsilon,t)=\left\{(\bp,q)\in\Z^{n+1}|~0<q<e^t, \inf\limits_{\bx\in\Delta\cap \bU}\left\Vert \f(\bx)-\frac{\bp}{q}\right\Vert_\infty<\frac{\varepsilon}{e^t}\right\},
\end{equation}
and $N(\Delta;\varepsilon,t):=\# \mathcal{R}(\Delta;\varepsilon,t).$
We refer the readers to the definition of  $\mathfrak{M}(\varepsilon,t)$ in \cite[Equation 4.19]{BY}. We take the  \begin{equation}\label{definition of minor} \begin{aligned}&
\text{ minor arcs as }\mathfrak{M}(e e^{\eta_1}\psi(e^{t-1}),t+\eta_2), \text{ and}\\
& \text{ major arcs as } \bU\setminus \mathfrak{M}(e e^{\eta_1}\psi(e^{t-1}),t+\eta_2).\end{aligned}\end{equation}

Next, let us define the following set:
\begin{equation}\label{minor} \mathcal{M}(t,\kappa)^{maj}:=\bigcup_{(\bp,q)\in\mathcal{R}(\bU\setminus\mathfrak{M}(e e^{\eta_1}\psi(e^{t-1}), t+\eta_2); e e^{\eta_1}\psi(e^{t-1}), t+\eta_2)}\left\{\bx\in\mathcal{M}(t,\kappa) ~|~\left\Vert \bx-\frac{\bp'}{q}\right\Vert_\infty<ee^{\eta_1}\frac{\psi(e^{t-1})}{e^{t+\eta_2}} \right\}.\end{equation}
Here, $\bp'=\pi(\bp)$, where $\pi:\R^n\to\R^d$ is the projection map.
Then it is easy to see, $$\mathcal{M}(t,\kappa)= \mathcal{M}(t,\kappa)^{maj} \bigcup  
\left ( \mathfrak{M}(e e^{\eta_1}\psi(e^{t-1}),t+\eta_2)\cap \mathcal{M}(t,\kappa)\right).$$

Therefore, it is enough to find dependency of $\kappa$ on $\delta$ such that  $$\lambda_d(\bigcup_{t\geq 1}\mathcal{M}(t,\kappa)^{maj})\leq \frac{\delta}{2}\lambda_d(\bU),$$ and $$\lambda_d(\bigcup_{t\geq 1}\left(\mathfrak{M}(e e^{\eta_1}\psi(e^{t-1}),t+\eta_2)\cap \mathcal{M}(t,\kappa)\right))\leq \frac{\delta}{2}\lambda_d(\bU).$$
Now the proof of \cref{main} will complete combining \cref{MainP_major} and \cref{mainp_Minor}.
\section{Some auxiliary lemmata}
We start with the following simple lemma.
\begin{lemma}\label{trivialbound}
 Let $\psi$ be monotonic nonnegative function. Then $$\sum_{q=1}^\infty\psi(q)^n<\infty\implies \psi(q)^n<\frac{L_\psi}{q}~~~ \forall q\in\N,$$ where $L_\psi:=\sum_{q=1}^\infty \psi(q)^n.$
\end{lemma}
\begin{proof}
Suppose we have $q\in \N$, then $$\sum_{k=1}^q \psi(k)^n<L_{\psi}\implies q\psi(q)^n<L_{\psi}\implies \psi(q)^n<\frac{L_\psi}{q}.$$
\end{proof}
Let us recall the definitions of the following matrices from \cite{BY}. We denote
$\J(\bx):=[\partial_{j} \ff_i(\bx)],$ and $\sigma_k\in M_{k\times k}$ be such that the off-diagonals are $1$, and other entries are $0$. We also denote $\mathbf{h}(\bx)=(\ff_1,\cdots,\ff_i)(\bx)-\J(\bx)\bx^T.$
Next we recall,

$$
u_1(\bx):=\begin{bmatrix}
\mathbb{I}_m & -\sigma_m^{-1}\J(\bx)\sigma_d & \sigma_m^{-1}\h(\bx)^T\\
0 & \mathbb{I}_d & \sigma_d^{-1}\bx^T\\
0 & 0 & 1
\end{bmatrix},
Z(\Theta):=\begin{bmatrix}
\mathbb{I}_m & \sigma_m^{-1}\Theta\sigma_d & 0\\
0 & \mathbb{I}_d & 0\\
0 & 0 & 1
\end{bmatrix}, \text{ where } \Theta\in M_{m\times d}.
$$
Also $$
U(\by):=\begin{bmatrix}
\mathbb{I}_n & \sigma_n^{-1}\by^T\\
0 & 1
\end{bmatrix}, \text{ where } \by\in\R^n.
$$
For $\bx\in\R^d$, $U(\bx):=U(\bx,\mathbf{0})$. For any $A>0, $ $U(A):=U(\by)$ for some $\by\in\R^n$ such that $\Vert\by\Vert_\infty\leq A.$ We similarly define $Z(A).$
The following lemma calculates explicit constants appearing in \cite[Lemma 4.5]{BY}.
\begin{lemma}\label{relationbetweenshift}
For any $\bx'\in\bU$ and $\bx\in\R^d$ such that the line segment joining $\bx'$ and $\bx+\bx'$ is contained in $\bU$. We have that \begin{equation}
    u_1(\bx+\bx')=Z(Md\Vert\bx\Vert_\infty)U(d^2M\Vert\bx\Vert^2_\infty)U(\bx)u_1(\bx').
\end{equation}
\end{lemma}
\begin{proof}

Using Taylor's series expansion  $\ff_i(\bx+\bx')=\ff_i(\bx')+ \sum_{j=1}^d \partial_j\ff_i(\bx+\bx')x_j+\sum_{\beta,\vert\beta\vert=2}\hat\ff_{i,\beta}(\bx')\bx_\beta,$ and $\lim_{\bx\to 0}\hat\ff_{i,\beta}(\bx+\bx')=0.$ Similarly using Taylor's series expansion, we have that $
\J_{i,j}(\bx+\bx')=\J_{i,j}(\bx')+ \sum_{k=1}^d \hat\J_{i,j,k}(\bx+\bx')x_k,$ where $\lim_{\bx\to 0}\hat\J_{i,j,k}(\bx+\bx')=0.$ Moreover, since the function $\f$ is $2$ times continuously differentiable in $\bU$, we also have an estimate on the error term, namely 
\begin{equation}\label{taylorconsequence}
\vert \hat\J_{i,j,k}(\bx+\bx')\vert \leq  \max_{\by\in\bU} \vert \partial_{k}\J_{i,j}(\by)\vert \text{ and }
\vert \hat\ff_{i,\beta}(\bx+\bx')\vert \leq \max_{\by\in\bU}\vert \partial_{\beta}\ff_i(\by)\vert.\end{equation}
Now $$
u_1(\bx+\bx')u_1(\bx')^{-1}U(\bx)^{-1}=
\begin{bmatrix}
\mathbb{I}_m & -\sigma_m^{-1}(\J(\bx+\bx')-\J(\bx'))\sigma_d & \sigma_m^{-1}(\h(\bx+\bx')-\h(\bx'))^T\\
0 & \mathbb{I}_d & 0\\
0 & 0 & 1
\end{bmatrix}
$$
The above matrix is same as the following product
$$\begin{bmatrix}
\mathbb{I}_m & -\sigma_m^{-1}(\J(\bx+\bx')-\J(\bx'))\sigma_d & 0\\
0 & \mathbb{I}_d & 0\\
0 & 0 & 1
\end{bmatrix}\begin{bmatrix}
\mathbb{I}_m & 0 & \sigma_m^{-1}(\h(\bx+\bx')-\h(\bx'))^T\\
0 & \mathbb{I}_d & 0\\
0 & 0 & 1
\end{bmatrix}. $$ Using \cref{Condition on derivative} and \cref{taylorconsequence} the lemma follows as the product of the above two matrices is $Z(Md\Vert \bx\Vert_\infty) U(d^2M\Vert \bx\Vert_\infty^2)$.
\end{proof}
For any $1>\varepsilon>0$ and $t>0$ the following diagonal matrix was defined in \cite[ Equation 4.3]{BY}, $$g_{\varepsilon,t}:=\diag\{\phi\varepsilon^{-1},\cdots,\phi\varepsilon^{-1},\phi e^{-t}\}\in \mathrm{SL}_{n+1}(\R), \text{ where } \phi:=\left(\varepsilon^n e^t\right)^{\frac{1}{n+1}}.$$ We also recall 
$$
b_t=\diag\{e^{\frac{dt}{2(n+1)}},\cdots,e^{\frac{dt}{2(n+1)}}, e^{\frac{-(m+1)t}{2(n+1)}},\cdots, e^{\frac{-(m+1)t}{2(n+1)}}, e^{\frac{dt}{2(n+1)}} \}\in \mathrm{SL}_{n+1}(\R).
$$
We recall the following lemmata from \cite{BY}.
\begin{lemma}\cite[Lemma 4.1]{BY}
Let $\by\in\R^n$. Then for any $t>0$, any $\Theta\in\R^{m\times d},$ if $\by\in B_{\infty}(\frac{\bp}{q},\frac{\varepsilon}{e^t})$ for some $(\bp,q)\in\Z^{n+1}$ with $0<q<e^t,$ then 
\begin{equation}
    \Vert g_{\varepsilon,t}Z(\Theta)U(\by) (-\bp\sigma_n,q)^T\Vert_\infty \leq c_0'\phi,\end{equation} where $$
    c_0'=\max_{1\leq i\leq m}(1+\vert \theta_{i,1}\vert+\cdots+\vert\theta_{i,d}\vert).$$
\end{lemma}

\begin{lemma}\cite[Lemma 4.2]{BY}\label{dani}
Let $\bx\in\bU.$ If $\f(\bx)\in B_\infty(\frac{\bp}{q},\frac{\varepsilon}{e^t})$ for some $(\bp,q)\in\Z^{n+1}$ with $0<q<e^t,$ then there exists $(\ba,b)\in\Z^{n+1}\setminus 0$ such that \begin{equation}
    \Vert g_{\varepsilon,t} u_1(\bx)(\ba,b)^T\Vert_\infty\leq c_1'\phi,\end{equation}
    where $c'_1=(d+1)M$.
\end{lemma}
\section{Counting points in Major arc}
In the following lemma we make the constants appearing in \cite[Lemma 5.5]{BY} explicit. Let us denote $\Delta_t(\bx_0):=\{\bx~|~\Vert \bx-\bx_0\Vert_\infty\leq \varepsilon e^{-t})^{\frac{1}{2}}\}$, and $\mathfrak{M}'(\varepsilon, t)=\bU\setminus \mathfrak{M}(\varepsilon, t).$
\begin{lemma}\label{counting}
Let a ball $B \subset \bU$ be given. Then for all $t >0$ and
all $\bx_0 \in \mathfrak{M}'(\varepsilon, t) \cap B$ we have that
\begin{equation}
    N(\Delta_t(\bx_0)\cap B;\varepsilon,t)\leq 2^{n+1} c_4^{(n+1)}c_3^m c_2^d\varepsilon^n e^t (\varepsilon e^{-t})^{-\frac{d}{2}}, 
\end{equation}
where $0<\varepsilon<1,$ $c'_1=(d+1)M,$ $c_2=(c'_1+1+d^2M),$ $c_3= c'_1+d^2M+d^2Mc_2$ and $c_4=\max\{c_3^{-1},c_2^{-1}, \frac{3}{2}\}.$

\end{lemma}
\begin{proof}
We follow the proof of \cite[Lemma 5.5]{BY}, but we explicitely compute the constants in each step.
So without loss of generality, let us take any $(\bp_1,q_1)$ such that there exists $\bx\in \Delta_t(\bx_0)\cap B$ such that $$
\left\Vert \f(\bx)-\frac{\bp_1}{q_1}\right\Vert_\infty<\frac{\varepsilon}{e^t}, 0<q_1<e^t.$$ 

By \cref{dani}, we have that \begin{equation}\label{Dani_supnorm}
   \Vert g_{\varepsilon,t}u_1(\bx)(-\bp\sigma_n,q)\Vert_\infty\leq c_1'\phi,
\end{equation} where $c_1'=(d+1)M$. By definition of $\Delta_t(\bx_0),$ we have that $\bx_0=\bx+(\varepsilon e^{-t})^{\frac{1}{2}}\bx',$ with $\Vert \bx'\Vert_\infty\leq 1.$ Now, by \cref{relationbetweenshift}, we have that 
$$
u_1(\bx_0)=Z(Md(\varepsilon e^{-t})^{\frac{1}{2}}\Vert \bx'\Vert_\infty)U(d^2 M \varepsilon e^{-t}\Vert \bx'\Vert_\infty^2)U((\varepsilon e^{-t})^{\frac{1}{2}}\bx') u_1(\bx).$$
By \cite[Lemma 4.3]{BY}, 
the right hand side of the above equation is 
$$ \begin{aligned}
&g_{\varepsilon,t}^{-1}Z(Md(\varepsilon e^{-t})^{\frac{1}{2}}\Vert \bx'\Vert_\infty)g_{\varepsilon,t}U(d^2 M \varepsilon e^{-t}\Vert \bx'\Vert_\infty^2)U((\varepsilon e^{-t})^{\frac{1}{2}}\bx') u_1(\bx)\\
&= g_{\varepsilon,t}^{-1}Z(Md(\varepsilon e^{-t})^{\frac{1}{2}}\Vert \bx'\Vert_\infty)U(d^2M\Vert \bx'\Vert_\infty^2) g_{\varepsilon,t} U((\varepsilon e^{-t})^{\frac{1}{2}}\bx') u_1(\bx)\\
&= g_{\varepsilon,t}^{-1}Z(Md(\varepsilon e^{-t})^{\frac{1}{2}}\Vert \bx'\Vert_\infty)U(d^2M\Vert \bx'\Vert_\infty^2)U((\varepsilon e^{-t})^{-\frac{1}{2}}\bx')g_{\varepsilon,t} u_1(\bx).
\end{aligned}$$

Hence, 
$$\begin{aligned}g_{\varepsilon,t}u_1(\bx_0)(-\bp\sigma_n,q)^T=&Z(Md(\varepsilon e^{-t})^{\frac{1}{2}}\Vert \bx'\Vert_\infty)U(d^2M\Vert \bx'\Vert_\infty^2)U((\varepsilon e^{-t})^{-\frac{1}{2}}\bx')g_{\varepsilon,t} u_1(\bx)(-\bp\sigma_n,q)^T\\
=& Z(Md(\varepsilon e^{-t})^{\frac{1}{2}}\Vert \bx'\Vert_\infty)U(d^2M\Vert \bx'\Vert_\infty^2)U((\varepsilon e^{-t})^{-\frac{1}{2}}\bx')\bv,\end{aligned}$$
where $
g_{\varepsilon,t}u_1(\bx)(-\bp\sigma_n,q)^T=\bv=(v_n,\cdots,v_1,v_0)^T.$
We write the above as, 
$$g_{\varepsilon,t}u_1(\bx_0)(-\bp\sigma_n,q)^T= Z(Md(\varepsilon e^{-t})^{\frac{1}{2}}\Vert \bx'\Vert_\infty)U(d^2M\Vert \bx'\Vert_\infty^2)\bv',$$ where $\bv'= U((\varepsilon e^{-t})^{-\frac{1}{2}}\bx')\bv.$

Suppose $[a]$ denotes the close interval $[-a,a]$. To begin with, note that using \cref{Dani_supnorm} and the fact that $\Vert\bx'\Vert_\infty\leq 1$, we have \begin{equation}\label{v'}
\bv'=(v_1',\cdots,v_{n+1}')\in [c'_1\phi]^m\times [(c'_1+1)\phi(\varepsilon e^{-t})^{-\frac{1}{2}}]^d\times [\phi],\end{equation} where $c'_1=(d+1)M.$ Now note that, 
$$\begin{aligned}
U(d^2M)\bv'=(v_1'+y_nv'_{n+1},\cdots, v_{m}'+y_{d+1}v'_{n+1},v_{m+1}'+y_{d}v'_{n+1},\cdots,v'_{n+1})^T,
\end{aligned}
$$ where $\by=(y_1,\cdots,y_n)\in\R^n$, and $\Vert\by\Vert_\infty\leq d^2M.$
Using \cref{v'}, we have the first $m$ components of $U(d^2M)\bv'$ is less than $(c'_1+d^2M)\phi$, and the last coordinate is bounded above by $\phi$. Next, note for $j=1,\cdots,d$, $$\begin{aligned}
&\vert v'_{m+j}+y_{d+1-j}v'_{n+1}\vert \leq (c'_1+1)\phi(\varepsilon e^{-t})^{-\frac{1}{2}}+d^2M\phi\\
& \leq (c'_1+1+d^2M)(\varepsilon e^{-t})^{-\frac{1}{2}}\phi.
\end{aligned}$$
The last inequality follows using, $
\varepsilon e^{-t}<1$ for all $t>0.$
Let us denote $c_2=(c'_1+1+d^2M).$ From the above calculation, we have that 
$$
U(d^2M)\bv'\in [(c'_1+d^2M)\phi]^m\times [c_2(\varepsilon e^{-t})^{-\frac{1}{2}}\phi]^d\times [\phi].$$

Next, we consider the action of $Z(Md(\varepsilon e^{-t})^{\frac{1}{2}})$ on $U(d^2M)\bv'=\bw=(w_1,\cdots,w_{n+1})^T,$
\begin{equation}
\begin{aligned}
 &Z(Md(\varepsilon e^{-t})^{\frac{1}{2}})\bw\\
& =(w_1+\theta_{m,d}w_{m+1}+\cdots+\theta_{m,1}w_n,\cdots,w_m+\theta_{1,d}w_{m+1}+\cdots+\theta_{1,1}w_n,w_{m+1},\cdots,w_{n+1})^T,
\end{aligned}
\end{equation}
where $\Vert \theta\Vert_\infty\leq Md(\varepsilon e^{-t})^{\frac{1}{2}}$.
Note that for $j=1,\cdots,m$, $$\vert w_j+\theta_{m+1-j,d}w_{m+1}+\cdots+\theta_{m+1-j,1}w_n\vert\leq (c'_1+d^2M)\phi+ d^2M  c_2\phi=(c'_1+d^2M+d^2Mc_2)\phi.$$
Therefore, we have that 
\begin{equation}
    Z(Md(\varepsilon e^{-t})^{\frac{1}{2}})\bw\in [c_3\phi]^m\times [c_2(\varepsilon e^{-t})^{-\frac{1}{2}}\phi]^d\times [\phi],
\end{equation}
where $c_3= c'_1+d^2M+d^2Mc_2. $
Then $$
b_tg_{\varepsilon,t}u_1(\bx_0)(-\bp\sigma_n,q)^T=b_tZ(Md(\varepsilon e^{-t})^{\frac{1}{2}})\bw\in [c_3\phi e^h]^m\times [c_2\varepsilon^{-\frac{1}{2}} e^h \phi]^d\times [\phi e^h],$$ where $h=\frac{dt}{2(n+1)}$.
Let us denote $$\Omega:=[c_3\phi e^h]^m\times [c_2\varepsilon^{-\frac{1}{2}} e^h \phi]^d\times [\phi e^h].$$
Since $\bx_0\in\mathfrak{M}'{(\varepsilon,t)},$ by definition of $\mathfrak{M}{(\varepsilon,t)}$; see \cite[Equation 4.19]{BY}, we have that $B(\mathbf{0},\phi e^h)$ contains a full fundamental domain of $b_tg_{\varepsilon,t}u_1(\bx_0)\Z^{n+1}.$ Here $B(\mathbf{0},\cdot)$ denotes ball with respect to $\Vert \cdot\Vert$ norm. Let us take $c_4=\max\{c_3^{-1},c_2^{-1}, \frac{3}{2}\}.$ Then $c_4c_3\phi e^h\geq \phi e^h$  and $c_4c_2\varepsilon^{-\frac{1}{2}} e^h \phi \geq  e^h\phi$, since $\varepsilon<1$. We get  $c_4>1$ such that $B(\mathbf{0},\phi e^h)\subset c_4\Omega.$ Therefore $c_4\Omega$ contains the full fundamental domain, which implies
$$
N(\Delta_t(\bx_0)\cap B;\varepsilon,t)\leq 2^{n+1} \mathcal{Lambda}_{n+1}(c_4\Omega) = 2^{n+1} c_4^{(n+1)}c_3^m c_2^d\phi^{n+1}\varepsilon^{-\frac{d}{2}}e^{(n+1)h}.$$

\end{proof}
The following lemma follows from combination of \cite[Lemma 5.4]{BY} and \cref{counting}.
 \begin{lemma}\label{countingproposition}
 Suppose $\bU\subset \R^d$ is an open ball and $\f:\bU\to\R^n$ is a $C^2$ map, both as in \cref{setup}. Then for any $0<\varepsilon<1$, ball $B\subset \bU$ and for all $t >0$ we have 
 \begin{equation}
     N(B\setminus \mathfrak{M}(\varepsilon,t),\varepsilon,t)\leq 2^{n+2} c_4^{(n+1)}c_3^m c_2^d \varepsilon^m e^{(d+1)t} \mathcal{Lambda}_d(B), \ \end{equation} 
 where the constants $c_2,c_3,c_4$ are same as in \cref{counting}.
 \end{lemma}
The following proposition estimates measure of major arc.
 \begin{proposition}\label{MainP_major}
 Let $\bU,\f$ be defined as in \cref{setup}. Then for any $0<\delta<1,$
 $$\lambda_d( \cup_{t\geq 1} \mathcal{M}(t,\kappa)^{maj})\leq \frac{\delta}{2}\lambda_d(\bU),\text { for }
 \kappa<\min\left(\left(\frac{1}{L_\psi e^{1+n}}\right)^{\frac{1}{5}},\frac{\delta}{2^{n+3} c_4^{(n+1)}c_3^m c_2^d e^m e^{d+1}S_\psi}\right),$$
 where  $S_{\psi}, c_2, c'_1, c_3, c_4$ are as in \cref{S_psi,c_2,c_1,c_3,c_4}, and $M$ is as in \cref{Condition on derivative}.
 
 \end{proposition}
 \begin{proof}
Let us take $\kappa^5<\frac{1}{L_\psi e^{1+n}}$. Then by \cref{kappacondition1}, and \cref{trivialbound} we have $e e^{\eta_1}\psi(e^{t-1})<1$ for any $t>0$. By \cref{countingproposition}, for $\varepsilon=e e^{\eta_1}\psi(e^{t-1})$ we have that $$\begin{aligned}
&N(\bU\setminus\mathfrak{M}(e e^{\eta_1}\psi(e^{t-1}),t+\eta_2); e e^{\eta_1}\psi(e^{t-1}), t+\eta_2)\\
&\leq 2^{n+2} c_4^{(n+1)}c_3^m c_2^d e^m e^{\eta_1 m}(\psi(e^{t-1}))^{m} e^{(d+1)(t+\eta_2)}\lambda_d(\bU).
\end{aligned}$$

Hence, 
\begin{equation}\begin{aligned}
&\lambda_d(\mathcal{M}(t,\kappa)^{maj})\\
&\leq  2^{n+2} c_4^{(n+1)}c_3^m c_2^d e^m e^{\eta_1 m}(\psi(e^{t-1}))^{m} e^{(d+1)(t+\eta_2)}\lambda_d(\bU) \left(\frac{e^{\eta_1}\psi(e^{t-1})}{e^{t+\eta_2-1}}\right)^d\\
& \leq  e^{\eta_1n+\eta_2} 2^{n+2} c_4^{(n+1)}c_3^m c_2^d e^m e^{d+1} \psi(e^{t-1})^n e^{t-1}\lambda_d(\bU).
\end{aligned}
\end{equation}
Therefore, 
$$\lambda_d( \cup_{t\geq 1} \mathcal{M}(t,\kappa)^{maj})\leq e^{\eta_1n+\eta_2} 2^{n+2} c_4^{(n+1)}c_3^m c_2^d e^m e^{d+1} \left(\sum_{t\geq 0}\psi(e^{t-1})^n e^{t-1}\right)\lambda_d(\bU).$$
Since $L_\psi=\sum_{q=1}^\infty \psi(q)^n<\infty$, we know $S_{\psi}:=\sum_{t=1}^\infty \psi(e^{t-1})^n e^{t-1}<\infty .$

We choose $$ e^{\eta_1n+\eta_2} 2^{n+2} c_4^{(n+1)}c_3^m c_2^d e^m e^{d+1} S_{\psi}<\frac{\delta}{2}\stackrel{\eqref{kappacondition2}}{\implies}\kappa= e^{\eta_1n+\eta_2}<\frac{\delta}{2^{n+3} c_4^{(n+1)}c_3^m c_2^d e^m e^{d+1}S_\psi}.$$
With the above choice, 
$$\lambda_d( \cup_{t\geq 1} \mathcal{M}(t,\kappa)^{maj})\leq \frac{\delta}{2}\lambda_d(\bU).$$
\end{proof}
 \section{Measure estimates of Minor arc}
For $\delta, K, T>0$, let us define 
\begin{equation}
    \mathcal{Delta}_{\f}(\delta,K,T):=\left\{\bx\in \bU~:~\exists (a_0,\ba)\in\Z\times \Z^n \text{ such that } \begin{aligned}
    &\vert a_0+\f(\bx)\ba^T\vert\leq \delta\\
    & \Vert \nabla\f(\bx)\ba^T\Vert_\infty<K\\
    & 0<\Vert \ba\Vert_\infty <T
    \end{aligned}
    \right\}.
\end{equation}
\begin{proposition}\label{mainp_Minor}
Let $\bU\subset \R^d$ be a ball and $\f$ be $C^{l+1}$ map, both with the assumptions as in \cref{setup}. For any $0<\delta<1,$
$$\lambda_d(\bigcup_{t\geq 1}\left(\mathfrak{M}(e e^{\eta_1}\psi(e^{t-1}),t+\eta_2)\cap \mathcal{M}(t,\kappa)\right))\leq \frac{\delta}{2}\lambda_d(\bU),$$ when 
$$\kappa<\min{\left( d_3^{-\frac{1}{4}}, e^{-\frac{n}{5}}d_3^{\frac{n}{5}}, \left(e^{-(n+1)}L_{\psi}^{-1}\right)^{\frac{1}{4n+5}}, d_3^{-(n+1)}e^{n}e^{\frac{5}{4}}, \left( \frac{\delta e^{nr+\frac{5r}{4}}}{2K_0c_r}\right)^{\frac{1}{r}} \right)},$$ where 
$d_3, K_0$, and $
E,$ are as in \cref{d3r,K_0 and E}, and $L_\psi$ is as in \cref{trivialbound}.

\end{proposition}
\begin{proof}
In the proof of \cite[Proposition 5.1]{BY}, it was shown that  $$
\mathfrak{M}(e e^{\eta_1}\psi(e^{t-1}),t+\eta_2)\subset \mathcal{D}_{\f}(\delta,K,T),$$ with $d_3= (n+1)!^2(1+n+n^3M)$, $\delta=d_3 e^{-t-\eta_2}$, $K=d_3 e^{-1} e^{-\eta_1}{\psi(e^{t-1})}^{-1} e^{-\frac{t+\eta_2}{2}}>0$, $ T=d_3
e^{-1} e^{-\eta_1}{\psi(e^{t-1})}^{-1}$. Let us take $\kappa<d_3^{-\frac{1}{4}}$, thus \begin{equation}\label{d_3 condition}4\log \kappa<-\log d_3\implies \log d_3<-4 \log\kappa=\eta_2\implies d_3<e^{\eta_2}\implies \delta<1.\end{equation} 
Let us also take $\kappa<e^{-\frac{n}{5}}d_3^{\frac{n}{5}}$, which implies \begin{equation}
    e e^{\eta_1}<d_3\implies T\geq 1.
\end{equation}
Next, observe that $$
\begin{aligned}
\frac{\delta K T^n}{T}=d_3^{n+1}e^{-n}e^{-n\eta_1}e^{-\frac{3\eta_2}{2}}e^{-\frac{3t}{2}}\psi(e^{t-1})^{-n},
\end{aligned}$$ and $\delta^{n+1}=d_3^{n+1}e^{-(t+\eta_2)(n+1)}.$
Let us choose $\kappa< \left(e^{-(n+1)}L_{\psi}^{-1}\right)^{\frac{1}{4n+5}}$, which implies that for any $t>0$,
$$\begin{aligned}
& e^{-(n+1)}\kappa^{-1}L_{\psi}^{-1}>e^{-\eta_2(n+1)}\stackrel{\eqref{kappacondition1}}{=}\kappa^{4(n+1)}\\
\implies &  e^{-(n+1)}\kappa^{-1}L_{\psi}^{-1}>e^{-\eta_2(n+1)} e^{-t(n+\frac{1}{2})}\\
\implies & e^{-(n+1)}\kappa^{-1}L_{\psi}^{-1} e^{-\frac{t}{2}}>e^{-\eta_2(n+1)} e^{-t(n+1)}\\
\stackrel{\eqref{kappacondition3}}{\implies} & e^{-n}e^{-\frac{3}{2}\eta_2} e^{-n\eta_1}\frac{L_{\psi}^{-1}}{e^{-(t-1)}} e^{-\frac{3t}{2}}>e^{-\eta_2(n+1)} e^{-t(n+1)}\\
\stackrel{\cref{trivialbound}}{\implies} & e^{-n}e^{-\frac{3}{2}\eta_2} e^{-n\eta_1}\psi(e^{t-1})^{-n} e^{-\frac{3t}{2}}>e^{-\eta_2(n+1)} e^{-t(n+1)}.
\end{aligned}
$$

The last inequality guarantees that $$
\delta^{n+1}<\frac{\delta K T^n}{T}.$$

Now let us choose $\kappa<d_3^{-(n+1)}e^{n}e^{\frac{5}{4}}$, which implies that for any $t>0$, $$\begin{aligned}
\frac{\delta K T^n}{T}\stackrel{\eqref{kappacondition3},\eqref{condition on si}}{<}d_3^{n+1}e^{-n}\kappa e^{-\frac{t}{4}}e^{-\frac{5}{4}}<1.
\end{aligned}$$ 
Therefore, by \cite[Theorem 5]{ABLVZ}, one gets that 
$$\lambda_d\left(\mathfrak{M}(e e^{\eta_1}\psi(e^{t-1}),t+\eta_2)\right)\leq K_0 \left(e^n e^{\eta_1n}\psi(e^{t-1})^n e^{\frac{3(t+\eta_2)}{2}} \right)^{-\frac{1}{d(2l-1)(n+1)}}\lambda_d(\bU),$$
where $K_0=d_3^{\frac{1}{d(2l-1)}}E(n+d+1)^{\frac{1}{2d(2l-1)}}$, and $$
E= C(n+1)(3^dN_d)^{n+1}\rho^{\frac{-1}{d(2l-1)}},$$ where $C$ can be found in \cite[Equation (52)]{ABLVZ}, $\rho$ is explicitly given in \cite[Equation (71)]{ABLVZ} and they only depends on $\f$ and $\bU$. 

Hence, we have that 
$$\lambda_d\left(\bigcup_{t\geq 1}\mathfrak{M}(e\kappa\psi(e^{t-1}),t)\cap \mathcal{M}(t,\kappa)\right)\leq K_0\sum_{t\geq 1}  \left(e^n e^{\eta_1n}\psi(e^{t-1})^n e^{\frac{3(t+\eta_2)}{2}} \right)^{-\frac{1}{d(2l-1)(n+1)}}\lambda_d(\bU).$$

Using \cref{condition on si}, note \begin{equation}\label{conditon on si 2}
 \psi(e^{t-1})^{-n}e^{\frac{-3t}{2}}<e^{-\frac{5}{4}} e^{-\frac{t}{4}}.\end{equation}
 Let $r=\frac{1}{d(2l-1)(n+1)}$. Then by \cref{conditon on si 2} and \cref{kappacondition3}, 
 $$
 K_0 e^{-nr}\left(e^{\eta_1n+\frac{3\eta_2}{2}}\right)^{-r}\left(\psi(e^{t-1})^{-n} e^{-\frac{3t}{2}}\right)^{r}<K_0 e^{-nr}\kappa^{r} e^{-\frac{5r}{4}}e^{-r\frac{t}{4}}.
 $$
Let $\sum_{t\geq 1} e^{-r\frac{t}{4}}=c_r.$ Let us choose $\kappa$ such that 
$$\begin{aligned}
&K_0 e^{-nr}\kappa^{r} e^{-\frac{5r}{4}} c_r<\frac{\delta}{2}\\
\implies & \kappa<\left( \frac{\delta e^{nr+\frac{5r}{4}}}{2K_0c_r}\right)^{\frac{1}{r}}.
\end{aligned}$$
\end{proof}

\subsection*{Acknowledgements} We thank Anish Ghosh, Subhajit Jana and Ralf Spatzier for several helpful remarks which have improved the presentation of this paper. We also thank MPIM, Bonn for warm hospitality, where majority of this work was done. 

\bibliographystyle{abbrv}
\bibliography{Quansim}

\end{document}